\documentclass[12pt,reqno]{amsart}
\usepackage{amsmath,amstext,amsfonts,amssymb}
\newtheorem{prop}{Proposition}[section]

\newtheorem{thm}[prop]{Theorem}
\newtheorem{lem}[prop]{Lemma}
\newtheorem{cor}[prop]{Corollary}
\theoremstyle{definition}

\newtheorem{rem}[prop]{Remark}

\newcommand{\bFp}{\mathbb{F}_p}

\begin{document}

\title[Expanding bounds]{Conditional expanding bounds for two-variable functions over prime fields}

\author[N. Hegyv\'ari]{Norbert Hegyv\'ari\textsuperscript{1}}

\address{Norbert Hegyv\'{a}ri, ELTE TTK,
E\"otv\"os University, Institute of Mathematics, H-1117
P\'{a}zm\'{a}ny st. 1/c, Budapest, Hungary}
\email{hegyvari@elte.hu}

\author[F. Hennecart]{Fran\c cois Hennecart\textsuperscript{2}}
\address{Fran\c cois Hennecart,
Universit\'e de Saint-\'Etienne,
Institut Camille Jordan,
23, rue Michelon,
42023 Saint-\'Etienne, France} \email{francois.hennecart@univ-st-etienne.fr}

\thanks{\textsuperscript{1} Budapest - Hungary}
\thanks{\textsuperscript{2} Saint-Etienne - France}

\thanks{Research of the authors is partially supported  by OTKA grants K~81658, K 100291}

\dedicatory{To Yahya}

\date{\today}

\begin{abstract}
In this paper we provide in $\bFp$ expanding lower bounds for  
two variables functions $f(x,y)$ in
connection with the  product set or the sumset. The sum-product problem 
has been hugely studied in the recent past. A typical result in $\bFp^*$
is the existenceness of $\Delta(\alpha)>0$ such that if $|A|\asymp p^{\alpha}$ then
$$
\max(|A+A|,|A\cdot A|)\gg |A|^{1+\Delta(\alpha)},
$$
Our aim is to obtain analogous
results for related pairs of two-variable functions $f(x,y)$ and $g(x,y)$:
if $|A|\asymp|B|\asymp p^{\alpha}$ then
$$
\max(|f(A,B)|,|g(A,B)|)\gg |A|^{1+\Delta(\alpha)}
$$
for some $\Delta(\alpha)>0$.
\end{abstract}

\maketitle
 \section{\bf Introduction}

We denote by $\bFp$ the field with $p$ elements and by $\bFp^*=\bFp\setminus\{0\}$ its multiplicative
group. Expanding properties of functions in $\bFp^*$ have been widely investigated in the
last decade. If $A\subset\bFp$, we denote by $|A|$ its cardinality and write $|A|\asymp p^{\alpha}$ if
$c_1p^{\alpha}<|A|<c_2p^{\alpha}$ for some fixed real numbers
$0<c_1<c_2$. Throughout the paper we will use the Vinogradov's symbol $\gg$  in the following way: $X\gg Y$ means that there exists
an absolute constant $\kappa>0$ such that $X\ge \kappa Y$ where
$X$ and $Y$ are numbers generally depending on certain parameters as the prime number $p$, the subsets $A,B,\dots$ of $\bFp$.

For a given function $f(x,y)$ and two subsets $A,B$ of $\bFp^*$,
we denote
$$
f(A,B)=\{f(a,b)\,:\, (a,b)\in A\times B\}.
$$ 
The sumset corresponds to the function $x+y$ and is denoted by $A+B$ ; the product-set corresponds to the function $xy$ and is denoted by $A\cdot B$.

A function $f:\bFp^*\times \bFp^*\to\bFp$ being given, what can be said on
$$
\inf_{\substack{A\subset \bFp^*\\|A|,|B|\asymp p^{\alpha}}}\frac{\log |f(A,B)|}{\ln p},
$$
for $0<\alpha<1$ ? A function is called an expander (according to $\alpha$) if
the above quantity is uniformly in $p$ bigger than $\alpha$. For instance, $f(x,y)=x(x+y)$ is known to be an
expander for any $\alpha$ (cf. \cite{Bo}).  A wide family of expanders has also
been provided in \cite{HH}.

A related question due to Erd\H os and  Szemer\'edi is the sum-product problem, which takes its roots from the analogous problem in $\mathbb{R}$. For $A$ be a finite subset in a ring, we denote 
$$
SP(A)=\max(|A+A|,|A\cdot A|).
$$ 
The best known statement  for real numbers asserts that  for $A\subset\mathbb{R}$,  
$$SP(A)\ge \frac{|A|^{4/3}}{2(\log|A|)^{1/3}},
$$ 
(see \cite{So2}).

\medskip
For a set $A$ in $\bFp^*$, the growth will be plainly limited according to the size of 
$\log|A|/\log p$. We may  confer \cite{BT} for a complete description of the recent improvements for the size of
 $SP(A)$. In particular for large subset $A$ of $\bFp$ Garaev (cf. \cite{Ga}) obtained the bound 
$SP(A)\gg \min(\sqrt{p|A|},|A|^2/\sqrt{p})$. His proof uses exponential sums. This result implies $SP(A)\gg |A|^{5/4}$ if $|A|\asymp p^{2/3}$. 
This bound has also been obtained in \cite{So1} by the use of a graph-theoretical approach. In the same paper  Solymosi proved 
also the bound  
$\max(|A+B|,|f(A)+C|)\gg \min(p|A|,|A|^2|B||C|/p)^{1/2}$
where $f$ is any polynomial with integral coefficients and degree
greater than one. In \cite{Vu} Vu introduces the class of non-degenerate
polynomials $f(x,y)$ over a finite field $\mathbb{F}$. 
For such a polynomial one has $\max(|A+A|,|f(A,A)|)\gg
\min(|A|^{2/3}|\mathbb{F}|^{1/3},|A|^{3/2}|\mathbb{F}|^{-1/4})$ for any
$A\subset \mathbb{F}$.
In \cite{HLS} Hart, Li and Shen studied such expanding phenomena
in connection with the sum-product property and 
obtained 
lower bound for the size of $\max(|u(A)*B|,|v(A)\circ C|)$ where
$u,v$ are polynomials over $\mathbb{F}$ and $*,\circ\in\{+,\times\}$.
All these 
lower bounds are non trivial only for $|A|,|B|>|\mathbb{F}|^{1/2}$.
In \cite{HLS}  the notion of expansion is 
also extended for subsets $E$
of $\mathbb{F}^2$ which are not necessarily a cartesian product 
$A\times B$  and gives a non trivial lower bound for
$\max(|f(E)|,|E+F|)$ where $f:\mathbb{F}^2\to\mathbb{F}$ is a non-degenerate polynomial of degree $k$
and $E,F$ are subsets of $\mathbb{F}^2$ with $|E|\gg k|\mathbb{F}|$.

For small subsets $A$ of the prime field $\mathbb{F}_p$, namely 
if $|A|\le\sqrt{p}$,
it has been proved in \cite{Li} that  $SP(A)\gg |A|^{13/12}$. We will use this fact in section \ref{s4}.  In \cite{She}, the author provides
a generalization of this lower bound by showing 
$\max(|A+A|,|f(A,A)|)\gg |A|^{13/12}$ for $|A|\le \sqrt{p}$
and where $f(x,y)=x(g(x)+y)$ for any arbitrary function $g:\bFp\to\bFp$. One will express this property by notifying that the function $f(x,y)$ satisfies a conditional expanding property relatively to the sum $x+y$. 

\medskip
All the above quoted results brought to light 
relative expansion properties according to a pair of two-variable functions, properties which are closely connected to the sum-product problem.
In this note, we are interested in a somewhat more general conditional expanding statement of the following
kind:
$$
|A|,|B|\asymp p^{\alpha}\Rightarrow \max(|f(A,B)|,|g(A,B)|)\gg|A|^{1+\Delta}.
$$
We will first obtain results with $g(x,y)=xy$ or $x+y$  combined with some more complicated functions $f(x,y)$. For it we will use a generalization of 
Solymosi's approach in \cite{So1} by the mean of $d$-regular graphs,
yielding to Theorems \ref{t1}, \ref{t2} and \ref{t3}.
Then we will use it
in connection with an explicit statement of the Balog-Szemer\'edi-Gowers Theorem
as done in \cite{GS} and \cite{BT} in the case $A=B$
(cf. Theorems \ref{t5} and \ref{t6}).
We will also focus our study on the function $xy(x^k+y^k)$
in section \ref{s4} for small subsets of $\mathbb{F}_p$ and
in section \ref{s6} for finite sets of real numbers.

\bigskip\noindent
\textit{Acknowledgement}.
The authors are greatful  to the referee for bringing into their 
knowledge some useful recent references on the subject.


\section{\bf Statement of the results}

We will use a bound  for the number of edges between two sets of vertices
of a regular directed graph. The idea of the proof is close to that from \cite{So1} and is extended to directed graph.  One first recall some basic facts on this notion. A (finite) directed
multigraph (or simply graph) 
$\Gamma=(V;E)$ is given by its (finite) set of vertices $V$ and
its (finite) set of edges $E\subset (V\times V)\times \mathbb{N}$ :
if in $\Gamma$, there are $k$ edges from $v$ to $w$, 
\textit{i.e.} $\left|\left(\{(v,w)\}\times\mathbb{N}\right)\cap E\right|=k$,
they are denoted $(v,w ; i)$, $i=1,\dots,k$. 
The adjacency matrix of $\Gamma$, $n=|V|$, 
is the $n\times n$ matrix
$M=(a_{vw})_{v,w \in V}$ defined by $a_{vw}=k$ 
if there are exactly $k$ edges from $v$ to $w$.
The (directed) graph $\Gamma$ is said to be $d$-regular if 
the sum of the coefficients on an arbitrary row or column of $M$ is equal to $d$.
It is said to be connected if, for each pair $(v,w)$ of vertices there exists a (directed) path in $\Gamma$ $v$ to $w$. It is said to be symmetric if  its adjacency matrix  $M$ is symmetric: it means that $\Gamma$
could be considered as a non directed multigraph.

For any arbitrary matrix $M$, we denote by $\mathstrut^{t}\! M$ its transpose. A matrix is said to be regular if it is the adjacency matrix of a regular (directed) graph. If $M$ is $d$-regular, then $N=\mathstrut^{t}\! MM$ is symmetric and $d^2$-regular. It means that $N$ can be viewed as
the adjacency matrix of a symmetric $d^2$-regular multigraph.

The main tool is the following discrepancy inequality which is similar to Theorem 9.2.5 of \cite[Chap. 9]{AS}. It will be proved in section \ref{ssp}.

\begin{thm}\label{th1}
Let $\Gamma=(V;E)$ be a $d$-regular (directed) graph with $n$ vertices and $M$ its adjacency matrix. 
We denote by $\theta_2$ the largest but the first one eigenvalue of $N=\,\mathstrut^{t}\! MM$. 

Then for any sets of vertices  $S$ and $T$ in $\Gamma$, one has 
\begin{equation}\label{eq0}
\Big|e(S,T)n-|S||T|d\Big|\le n\sqrt{\theta_2 |S||T|},
\end{equation}
where $e(S,T)$ is the number of edges in $\Gamma$ from $S$ to $T$.
\end{thm}

The main points in  the above theorem are:\\
\indent -- $\mathstrut^{t}\! MM$ is a real symmetric matrix for which we can use 
the spectral theorem, namely its eigenvalue are positive real numbers
and its eigenspace are orthogonal,\\
\indent -- $\mathstrut^{t}\! MM$ is the adjacency matrix of a symmetric
regular multigraph.

\bigskip
For a given subgroup $G$ of $\bFp^*$ and $g:G\to\bFp$ an arbitrary function, we define
$$
\mu(g)=\max_{t}|\{ t=g(x)\,: x\in G \}|.
$$

By the incidence theorem due to Bourgain, Katz and Tao (cf. \cite{BKT}),
we can show that the function $f(x,y)=g(x)(h(x)+y)$ is an expander,
namely $|f(A,B)|\gg |A|^{1+\Delta(\alpha)}$, 
for any $A,B$ such that $|A|,|B|\asymp p^{\alpha}$ with $0<\alpha<1$, whenever $g$ and $gh$ are sufficiently affinely independent: it means that
we can control the number of solutions of the system
\begin{equation}\label{ss1}
\begin{cases}
\mu g(x')-\lambda g(x)=0\\
 \lambda h(x') -\mu  h(x)= 0 \text{ or }1,
\end{cases}
\end{equation}
for any choice of the pair $(\lambda,\mu)$.
A similar expanding property holds for the more general function
$f(x,y)=g(x)(h(x)+k(y))$ where $\mu(k)=O(1)$. 
Unfortunately, the growth exponent  $\Delta(\alpha)$ is rather weak when $\alpha\le 1/2$: in \cite{HR}  the authors have shown an explicit exponent for the Bourgain-Katz-Tao incidence Theorem which yields
 the admissible exponent $\Delta(\alpha)=1/10678$. 
If 
$1/2<\alpha<1$, then by  Vinh's incidence theorem (cf. \cite{Vi}), the growth exponent  
can be chosen equal to $\Delta(\alpha)=\min(1-1/2\alpha,(1/\alpha-1)/2)$ since for $|A|\asymp |B|$
$$
|f(A,B)|\gg \min\left(\frac{|A|^2}{\sqrt{p}},\sqrt{p|A|}\right).
$$

Our aim is to obtain a more uniform result with a conditional better growth exponent under some additional assumption on the size of the product-set or the sumset.  Moreover the admissible functions $f$ are not necessary rational functions,
contrarily to the most studied cases in the literature. 

\medskip
Our first result  which will be proved in section \ref{s2} is the following.
It  generalizes Theorem 2 of \cite{GS}, obtained 
by Fourier analysis. If one compares it with
Theorem 2.6  of \cite{HLS}, we could observe that
the our result concerns also functions $f(x,y)$ in which the variables cannot be additively nor multiplicatively separated in the sense that 
it cannot be written under the forms $u(x)+v(y)$ or $u(x)v(y)$.  

\begin{thm}\label{t1}
 Let $G$ be a subgroup of $\bFp^*=\bFp\setminus\{0\}$ and 
$f(x,y)=g(x)(h(x)+y)$ be defined on $G\times \bFp^*$ where $g,h:G\to \bFp^*$ 
are arbitrary functions.  Put $m=\mu(g\cdot h)$. 
For any sets $A\subset G$ and $B,C\subset \bFp^*$,
we have
$$
|f(A,B)||B\cdot C|\ge \frac18\min\left(\frac{|A||B|^2|C|}{pm^2},
\frac{p|B|}{m}\right).
$$
\end{thm}

Letting $C=A$ in this theorem, we get 
$$
|A|,|B|\asymp p^{\alpha}\Rightarrow \max(|f(A,B)|,|A\cdot B|)\gg|A|^{1+\Delta(\alpha)},
$$
where 
$\Delta(\alpha)=\min(1-1/2\alpha,(1/\alpha-1)/2)$.

Observe that the condition on $g$ and $h$ is  different from the one requested when applying
incidence inequality. In fact our results may hold and in a same time 
$f$ is not an expander. Moreover the validity of this result  weakly  depends
on the functions $g$ and $h$. For instance, our result holds
for $f(x,y)=x(1+y)$, which is not an expander, as it can be easily observed.
On an other hand, in the restricted case $A=B$, Garaev and Shen proved in \cite{GS} that $|A\cdot(A+1)|\gg \min(\sqrt{p|A|},|A|^2/\sqrt{p})$. For small $|A|$, that is $|A|\le \sqrt{p}$,
they avoided the use of Bourgain-Katz-Tao inequality 
and obtained the bound $|A\cdot(A+1)|\gg|A|^{106/105+o(1)}$
by an argument of Glibichuk and Konyagin (cf. \cite{GK}), some ingredients from \cite{KS} and \cite{BG} and the use of an explicit Balog-Szemer\'edi-Gowers type estimate.

Furthermore for  $f(x,y)=1+xy=x(1/x+y)$, then clearly $f$ is not an expander and Theorem \ref{t1} holds but gives a trivial bound:
 in such a case $|f(A,B)|$ and $|B\cdot C|$ 
can be simultaneously small, namely if $A=B=C$ is a geometric progression in $\bFp^*$. 

By applying this bound with $k(B)$ instead of $B$, we get a similar bound for
the function $f'(x,y)=g(x)(h(x)+k(y))$ if we assume $k$ to be one to one, since,
in that case $|k(B)|=|B|$. With no additional assumption on $k$, we only have
$$
|f'(A,B)||k(B)\cdot C|\ge \frac18\min\left(\frac{|A||B|^2|C|}{p\mu(gh)^2\mu(k)^2},
\frac{p|B|}{\mu(gh)\mu(k)}\right),
$$
since $|k(B)|\ge|B|/\mu(k)$. If one takes $C=k(C')$ with $k(x)=x^j$, we obtain
a lower bound involving the product $B\cdot C'$.

\medskip
The next result is the additive counterpart of Theorem \ref{t1}. It will
be proved in section \ref{s3} by considering the simple
sum-product graph (cf. \cite{So1}).  

\begin{thm}\label{t2}
Let $G$ be a subgroup of $\bFp^*=\bFp\setminus\{0\}$ and $f(x,y)=g(x)(h(x)+y)$ be defined on $G\times G$ where $g$ and $h$ are arbitrary 
functions from $G$ into $\bFp^*$.
 Put $m=\mu(g)$.
For any $A\subset G$, $B,C\subset\bFp^*$, we have
$$
|f(A,B)||B+C|\gg \min\left(\frac{p|B|}m,\frac{|A||B|^2|C|}{pm^2}
\right).
$$
\end{thm}
By letting $C=A$, this yields  
$$
|A|,|B|\asymp p^{\alpha}\Rightarrow
\max(|f(A,B)|,|A+B|)\gg|A|^{1+\Delta(\alpha)},
$$
where 
$\Delta(\alpha)=\min(1-1/2\alpha,(1/\alpha-1)/2)$.

Notice that Theorem 6 in \cite{BT} does not cover such a function like in Theorems \ref{t1} and \ref{t2}. Observe also that when $g$ and $h$ are polynomials and $g$ is non constant, Vu's estimate in \cite{Vu} or its generalization in \cite{HLS} (cf. Theorem 2.9) would lead to a similar statement with a weaker exponent $\Delta(\alpha)=\min((2-1/\alpha)/4,(1/\alpha-1)/3)$.

\medskip
It is not known whether or not the function $f(x,y)=xy(x+y)$ is an expander in $\bFp$ while in the same time it can be shown to be an expander in $\mathbb{R}$ (see section \ref{s6}).
This question is seemingly hard to tackle with actually known tools. But somewhat surprisingly, the related functions
$(x+y)/xy=1/x+1/y$ and $xy+x+y=(1+x)(1+y)-1$ are plainly not  expanders.

Our aim
is to obtain a conditional expanding lower bound as in Theorem \ref{t1}
for this kind of function. We state such a result which involves
in addition $k$-th powers residues. For any function $h$, we denote $h_u(x)=h(ux)$.

\begin{thm}\label{t3}
Let $f(x,y)=g(x)h(y)(x^k+y^k)$ where $g,h:G\to\bFp^*$ are functions defined on some
subgroup $G$ of $\bFp^*$. We assume that for any fixed $z\in G$, $g(xz)/g(x)$ 
and $h(xz)/h(x)$ take $O(1)$  different values when $x\in G$ and that
$\max_u\mu(g\cdot h_u\cdot \mathrm{id})=O(1)$.
Then for
any $A,B,C\subset G$, one has
$$
|f(A,B)||A\cdot C||B\cdot C|\gg  
\min\left(\frac{|A|^2|B|^2|C|}{p},p|A||B|\right).
$$ 
\end{thm}The above theorem will be proved in section \ref{s4}.
The condition on $g$ and $h$ in the theorem looks unusual. For instance, $g$ and $h$ could be 
monomial functions. Other examples are given by functions $\tau^{\alpha(x)}x^k$ where 
$\tau\in\bFp^*$ has order $O(1)$ and $\alpha$ is an arbitrary function. If $g$ and $h$ are both
multiplicative homomorphisms, then, we can improve this result (see inequality  \eqref{nnn1}). Theorem \ref{t3} applies to the polynomial $f(x,y)=xy(x^k+y^k)$ where $k$ is a positive integer.
We will also consider in section the case of small subsets
of $\mathbb{F}_p$ (cf. Theorems \ref{t7} and \ref{t8}).

\medskip
Let $f(x,y)$ and $w(x)$  defined for $x,y\in \bFp^*$. We define the multiplicative  $w$-shifted function of $f$  by $P_w(f)(x,y)=w(x)f(x,y)$. We have

\begin{thm}\label{t5} Let $f(x,y)$ as in Theorem \ref{t3} and $A\subset \bFp^*$. Let $w$ be a function such that $\mu(w)=O(1)$. 
Then for any $1/2<\alpha<1$, there exists $\Delta=\Delta(\alpha)>0$ 
 such that
\begin{align*}
&|A|\asymp p^{\alpha}\Rightarrow \max(|f(A,A)|,|P_w(f)(A,A)|)\gg |A|^{1+\Delta}.
\end{align*}
\end{thm}

In a similar way, we may define the additive $w$-shifted function of $f$ by
$S_w(f)(x,y)=w(x)+f(x,y)$. We have

\begin{thm}\label{t6} Let $G$ and  $f(x,y)=g(x)(h(x)+y)$  with
$g$ and $h$ as in  Theorem \ref{t2} and $A\subset G$. Let $w$ such that $\mu(w)=O(1)$. 
Then for any $1/2<\alpha<1$, there exists $\Delta=\Delta(\alpha)>0$ 
 such that
\begin{align*}
&|A|\asymp p^{\alpha}\Rightarrow \max(|f(A,A)|,|S_w(f)(A,A)|)\gg 
|A|^{1+\Delta}.
\end{align*}
\end{thm}
Both of these results are proved in section \ref{s7}.

\medskip
 By a geometrical approach coming from \cite{So2},
we will obtain a conditional bound  for $f(x,y)=xy(x^k+y^k)$ in the real numbers
for arbitrary $k\in\mathbb{R}\setminus\{0\}$
(cf. Proposition \ref{pp71}).
Finally thanks to an appropriate extension of Szemer\'edi-Trotter Theorem
to curves (see \cite{TV}, Theorem 8.10) we will finally obtain
in section \ref{s6} expanding results for $f(x,y)=xy(x+y)$
in the real numbers
(cf. Proposition \ref{pp73}).

\section{\bf Spectral properties of regular graphs}\label{ssp}

In this section, we collect some known results. 
The assumption on the regularity of $\mathstrut^{t}\! MM$
is the key point to ensure that analogous properties
of the sum-product graph described by Solymosi still hold.
We first recall a result giving spectral properties for symmetric regular multigraphs. Namely it identifies
strictly the largest eigenvalue in absolute value and also its associated eigenspace. For the sake of completeness we
include the proofs.

A graph is said to be simple if the coefficients of its adjacency matrix
are  0 or 1.
The adjacency matrix of a simple  $d$-regular graph is called a simple $d$-regular matrix.

\begin{lem}\label{lem10}
Let $n,d\ge2$ be integers.
Let $M$ be the adjacency $n\times n$-square matrix of a 
$d$-regular connected multigraph $\Gamma$.
Then $d$ is the largest eigenvalue (in absolute value) of $M$ and 
its eigenspace is the line $\mathbb{C}\mathbf{1}$. 

Moreover for any eigenvalue $\lambda$ different from
$d$, any eigenvector $v$ associated to $\lambda$ is orthogonal to the $n$-vector $\mathbf1$.
\end{lem}

\begin{proof}\mbox{}\\
1) We first prove the result for simple graphs.
In that case  $M$ is a simple $d$-regular matrix.
It is clear that $d$ is one of its eigenvalue and that the $n$-vector $\mathbf1=(1,1,\dots,1)$ is an associated eigenvector. Let $v$ be an eigenvector for another
eigenvalue $\lambda$. Then
\begin{equation}\label{gg1}
\sum_{k=1}^d v_{jk}=\lambda v_j,\quad j=1,\dots,n,
\end{equation}
with the property that each $v_i$ appears exactly $d$ times among the elements
$v_{jk}$, $j=1,\dots,n$, $k=1,\dots,d$,
since $M$ is $d$-regular.

By summing, we obtain
$$
d\sum_{j=1}^nv_j=\lambda\sum_{j=1}^nv_j.
$$
Hence either $\lambda=d$ or $\sum_{j=1}^nv_j=0$. In this
later case, the vector $v$ is orthogonal to $\mathbf{1}$.

Moreover from \eqref{gg1} we have
$$
|\lambda|\sum_{j=1}^n|v_j|=\sum_{j=1}^n\Big|\sum_{k=1}^dv_{jk}\Big|
\le \sum_{j=1}^n\sum_{k=1}^d |v_{jk}|=d\sum_{j=1}^n|v_j|,
$$
hence $|\lambda|\le d$. We now assume that $\lambda=d$. From the preceding inequalities,
we get 
$$
d|v_j|=\Big|\sum_{k=1}^dv_{jk}\Big|=\sum_{k=1}^d |v_{jk}|,\ j=1,\dots,n.
$$
Since the graph is connected, it implies first that all the $v_j$ have the same absolute value.
Each equation in \eqref{gg1} where $\lambda=d$ shows that $d$ is the sum of $d$
roots of unity $v_{jk}/v_j$, $k=1,\dots,d$. Hence $v_{jk}=v_j$ for any $k=1,\dots,d$.
By the connectivity of the graph, it implies that all the $v_j$ are equal.
 It follows that the eigenspace associated to the eigenvalue $d$ is exactly the line $\mathbb{C}\mathbf{1}$.

\bigskip
\noindent
2)
Let $M=(a_{ij})_{i,j=1}^n$. We have $a_{ij}\in\mathbb{N}$ and $\sum_{k=1}^na_{ik}=d$ for any
$i=1,\dots,n$.
According to the same argument as in the first case, it is enough to prove that the eigenspace associated to the
maximal eigenvalue (namely $d$) is $\mathbb{C}\mathbf{1}$. If $v$ is an eigenvector for $d$ then
\begin{equation}\label{nn1}
\sum_{k=1}^na_{ik}v_k=dv_i,\quad i=1,\dots,n,
\end{equation}
and in the same way as above, we have 
$$
d|v_i|=\Big|\sum_{k=1}^na_{ik}v_k\Big|=\sum_{k=1}^na_{ik}|v_k|, \quad i=1,\dots,n.
$$
We deduce that the $v_k$'s for which $a_{ik}\ne0$ have the same argument than $v_i$.  Since the
graph is connected, all the $v_i$'s have the same argument. We may assume that $v_i\in\mathbb{R}^+$,
$i=1,\dots,n$.  Rearrange the $v_i$'s in the increasing order. Without loss of generality, we may assume
that $0\le v_1\le \cdots\le v_n$. We clearly obtain from \eqref{nn1} that $v_1=v_{k}$, for all $k$
such that $a_{ik}\ne0$. By minimality of $v_1$ and by connectivity of the graph, all the $v_i$'s are equal,
as asserted.
\end{proof}

We will also need an appropriate  bound for the number of edges between two sets of vertices in 
a directed regular graph.
The next lemma allows us to obtain Theorem \ref{th1} in a straightforward way as in the book of Alon and Spencer
\cite{AS}.

\begin{lem}\label{lem11}
Let $M$ be an $n\times n$ matrix and put $N=\mathstrut^{t}\! MM$.
We denote by $\theta_1\ge\theta_2\ge \cdots\ge \theta_n\ge0$ 
 the $n$ eigenvalues of the positive symmetric matrix $N$. Let $w$
be an eigenvector for $N$ associated to the eigenvalue $\theta_1$.
Then for any vector $v$ orthogonal to the $n$-vector $w$, we have
$$
|(Mv,Mv)|\le \theta_2\|v\|^2
$$
where $(\cdot,\cdot)$ denotes the canonical scalar product in $\mathbb{R}^n$ and $\|v\|=\sqrt{(v,v)}$ denotes the euclidean norm of $v$.
\end{lem}

\begin{proof}
By the spectral theorem we write $v=u_2+\cdots+u_n$ where $u_i$ is an eigenvector associated to $\theta_i$, $i=2,\dots,n$. We also know
by  that the $u_i$'s form an orthogonal family. We have
$Nv=\sum_{i=2}^n\theta_iu_i$, hence
$$
|(Mv,Mv)|=|\mathstrut^{t}v Nv|=\left|\sum_{i=2}^n\theta_i\|u_i\|^2\right| \le \theta_2\sum_{i=2}^n\|u_i\|^2=\theta_2\|v\|^2.
$$
\vspace{-1.3cm}\[\qedhere\]
\end{proof}

\bigskip
Finally we show a general bound for a  $d$-regular simple graph:

\begin{lem}\label{lem2}
Let $M$ be a  simple $d$-regular $n\times n$ matrix. Then for any vector $v$, one has
$$
|\mathstrut^{t}vMv|\le d\|v\|^2.
$$
\end{lem}

\begin{proof}
We have
$$
\mathstrut^{t}vMv=\sum_{j=1}^nv_j(v_{j1}+\cdots+v_{jd}),
$$
with the property that each $v_i$ appears exactly $d$ times among the elements
$v_{jk}$, $j=1,\dots,n$, $k=1,\dots,d$.

By the Cauchy-Schwarz inequality, we obtain
$$
|\mathstrut^{t}vMv|\le \sum_{k=1}^d \sum_{j=1}^nv_jv_{jk}\le \|v\| \sum_{k=1}^d V_k,
$$
where $V_k=(\sum_{j=1}^n v_{jk}^2)^{1/2}$. Again by the Cauchy-Schwarz inequality, 
$$
|\mathstrut^{t}vMv|\le \|v\| \sqrt{d}\left(\sum_{k=1}^d\sum_{j=1}^n v_{jk}^2\right)^{1/2}.
$$
By our assumption, the double summation equals $d\|v\|^2$. Hence the result.
\end{proof}

We now prove Theorem \ref{th1}.

\begin{proof}[Proof of Theorem \ref{th1}]
Let $V=\{1,2\dots,n\}$ be the vertices of $\Gamma$.
We let $s=|S|/n$ and $t=|T|/n$. Consider the vector 
$v=(v_1,\dots,v_n)$ defined by
$$
v_i= \begin{cases}  1-s &\text{si $i\in S$,}\\
-s & \text{otherwise.}
\end{cases}
$$
Then $\sum_{i=1}^nv_i=0$, thus $v$ is orthogonal to
$\mathbb{C}\mathbf{1}$ which is, by Lemma \ref{lem10}, 
the eigenspace associated to $d^2$, that is the largest
eigenvalue of $N=\mathstrut^{t}\! MM$. Then by Lemma
\ref{lem11}, we get the bound $\|Mv\|^2\le \theta_2\|v\|^2$.
We have $\|v\|^2=sn(1-s)^2+(n-sn)s^2=ns(1-s)$.
By letting $N_S(i)$ the set of vertices $j\in S$ such that $(i,j)\in E$,
the $i${th} coordinate of $Mv$ is equal to
$(1-s)|N_S(i)|-s(d-|N_S(i)|)$, thus
$$
\|Mv\|^2=\sum_{i\in V}\big((1-s)|N_S(i)|-s(d-|N_S(i)|)\big)^2
=\sum_{i\in V}(|N_S(i)|-sd)^2.
$$
Hence
\begin{equation*}
\sum_{i\in V}(|N_S(i)|-sd)^2\le \theta_2s(1-s)n\le \theta_2|S|,
\end{equation*}
which plainly implies
\begin{equation}\label{eqss}
\sum_{i\in T}(|N_S(i)|-sd)^2\le \theta_2|S|.
\end{equation}
This gives
$$
\left|e(S,T)-\frac{|S||T|d}n\right|\le \sum_{i\in T}\Big||N_S(i)|-sd\Big|
\le \sqrt{|T|}\left(\sum_{i\in T}\big(|N_S(i)|-sd\big)^2\right)^{1/2}
$$
by the triangle inequality and the Cauchy-Schwarz inequality.
By \eqref{eqss} the result follows. 
\end{proof}

\section{\bf Conditional expansion by the function $\boldsymbol{f(x,y)=g(x)(h(x)+y)}$ relatively to product function $\boldsymbol{xy}$}\label{s2}

Let $G$ be a subgroup of $\bFp^*$.
Here we assume that $g$ and $h$ are arbitrary functions
$$ 
g,g',h:G\to \bFp^*.
$$
For $A\subset G$, 
we are interested in the growth of the set $f(A,B)=\{g(x)g'(y)(h(x)+y)\,:\, x\in A,y\in B\}$ in terms of $|A|$ and $|B|$.
In this section we are mainly interested by the proof of 
Theorem \ref{t1} which refers to the case $g'=1$. 
But in section \ref{s4} the function $g'$ will be considered in all its generality.

For $a,b,c,d$ in $\bFp$, we first consider the number of solutions $N(a,b,c,d)$ to  the system
$$
\begin{cases}
ax=g(b)g'(y)(b+y)\\
cx=g(d)g'(y)(d+y).
\end{cases}
$$ 
%
It implies
$$
(cg(b)-ag(d))y=(adg(d)-bcg(b)).
$$
We distinguish the following cases:

\begin{itemize}
\item[--] If $cg(b)-ag(d)=adg(d)-bcg(b)=0$, then $(a,c)=(b,d)$
and each $y$ different from $-b$ gives a solution $(x,y)$:
$N(a,b,c,d)=p-2$.


\item[--]  If $cg(b) = ag(d)$ and $adg(d)\ne bcg(b)$, then $b\ne d$
and  $N(a,b,c,d)=0$. For a fixed $(a,b)$, there are $p-2$ couples $(c,d)$ with this property.

\item[--]  If $cg(b) \ne ag(d)$ and $bcg(b)=adg(d)$, then $b\ne d$
and $N(a,b,c,d)=0$.
 For a fixed $(a,b)$, there are $p-2$ couples $(c,d)$
with this property.

\item[--]  Assume now that 
\begin{equation}\label{eq2}
cg(b) \ne ag(d)\quad \text{and}\quad adg(d)\ne bcg(b).
\end{equation} 
Then
$(x,y)$ is solution if and only if 
$$
y=\frac{adg(d)-bcg(b)}{cg(b)-ag(d)},\quad
x=\frac{g(b)g(d)(d-b)}{cg(b)-ag(d)}g'\left(\frac{adg(d)-bcg(b)}{cg(b)-ag(d)}\right).
$$
Hence $N(a,b,c,d)=1$ if $d\ne b$, and $0$ otherwise.
In particular, $N(a,b,c,b)=0$ whenever  $c\ne a$ and $d=b$.
Hence there are $p-2$ more couples $(c,d)$ for which  $N(a,b,c,b)=0$.
\end{itemize}
For summarizing, for each $(a,b)$ there is one couple $(c,d)=(a,b)$
for which $N(a,b,c,b)=p-2$, $3p-6$ couples $(c,d)$ such that
$N(a,b,c,b)=0$ and for the $(p-1)^2-(3p-6)-1$ remaining couples $(c,d)$
we have $N(a,b,c,b)=1$.

\medskip
Let $\Gamma= (\bFp^*\times\bFp^*;\mathcal{E})$ be the (directed) graph defined by
$$
((a,b);(c,d))\in \mathcal{E}\iff ac=g(b)(b+d).
$$
We denote by $M$ its adjacency matrix. Notice that $M$ is 
generally not symmetric but  
$\Gamma$ is a $(p-2)$-regular directed graph:  indeed
$(a,b)$ being given, for any $d\ne -b$, there exists exactly one admissible $c$. Similarly for given $(c,d)$, there is exactly one
admissible $a$ for any $b\ne -d$.

\noindent
The product $\mathstrut^t\!MM$ is exactly the matrix whose coefficients are
the numbers $N(a,b,c,d)$, defining for two vertices
$(a,b)$ and $(c,d)$ of $\Gamma$. $\mathstrut^t\!MM$ is
a $(p-2)^2$-regular symmetric matrix.

\noindent
By the above computations on $N(a,b,c,d)$  we get
\begin{equation}\label{eq1}
\mathstrut^t\!MM=J+(p-3)I-E,
\end{equation}
where $J$ is the $(p-1)^2\times(p-1)^2$-square matrix composed uniquely by $1$, and  $E$ is the error matrix. 
Its coefficients are $0$ or $1$. Moreover, $E$ is symmetric, simple and $(3p-6)$-regular, in the sense that each row and each column
contain exactly $3p-6$ ``1", the other coefficients being $0$. 

We may
check that this multigraph is connected for $p$ enough large.  
Indeed, since for each row and each column, all but $3p-6<(p-1)^2/2$ coefficients of $\mathstrut^t\!MM$ are positive
(for $p>5$), it follows that all the coefficients of $(\mathstrut^t\!MM)^2$ are positive. In other words 
 each pair of vertices can be connected by a directed path formed by at most $2$ edges.

By the spectral theorem the eigenvalues of $\mathstrut^t\!MM$
are nonnegative real numbers.
We now apply Lemma \ref{lem10} to $\mathstrut^t\!MM$. 
We first obtain  that $(p-2)^2$ is the largest eigenvalue of $\mathstrut^t\!MM$. 
Let $(p-2)^2=\theta_1\ge\theta_2\ge\cdots\ge\theta_n\ge0$ be  
the eigenvalues of $\mathstrut^t\!MM$ and
$v$ be an eigenvector of $\mathstrut^t\!MM$ associated to the eigenvalue $\theta_2$.
By Lemma \ref{lem10} also, we get that $v$ is orthogonal to $\mathbf{1}$. Hence $Jv=0$. 
By Lemma \ref{lem2} we have $|\mathstrut^{t}vEv|\le (3p-6)\|v\|^2$. Hence from \eqref{eq1}, we deduce
$$
|\theta_2|\|v\|^2=|\mathstrut^{t}v\mathstrut^{t}\! MMv|\le 
(p-3)\|v\|^2+ |(v,Ev)|\le (4p-9)\|v\|^2,
$$
which gives $|\theta_2|\le 4p-9$.

We are in position to prove Theorem \ref{t1}.

\begin{proof}[Proof of Theorem \ref{t1}]
We now fix $A\subset G$ and $B\subset G$.
We may  apply  Theorem \ref{th1} with an appropriate choice of the sets of vertices $S$ and $T$. Let 
$$
m=\mu(g\cdot h).
$$

We define 
\begin{align*}
S&=\{(zg(zh(x))/g(x),zh(x)),\ (x,z)\in A\times C\},\\
T&=\{(g(x)(h(x)+y),yz),\ (x,y,z)\in A\times B\times C\}.
\end{align*}
We plainly have $|T|\le \min(|f(A,B)||B\cdot C|,|A||B||C|)$ and $|S|\le |A||C|$. Furthermore
for a  given quadruple $(u,v,w,t)$, the number of solutions $(x,y,z)$ to the system
\begin{equation}\label{eqca}
\begin{cases}
g(x)(h(x)+y)=u\\
yz=v\\
zg(zh(x))/g(x)=w\\
zh(x)=t,
\end{cases}
\end{equation}
is at most $m$. Indeed one easily check that these conditions imply
$y=vh(x)/t$, $g(x)=zg(t)/w$, $h(x)=t/z$ hence
$g(x)h(x)=tg(t)/w=ut/(v+t)$. By our assumption on $\mu(g\cdot h)$ there
are at most $m$ different values of $x$ and $(y,z)$ is uniquely 
determined in terms of $x$ by the second and the fourth equations of \eqref{eqca}.
Thus
$$
e(S,T)\ge\frac{|A||B||C|}{m}.
$$
By letting 
$X=\sqrt{|S||T|}$
we thus obtain from \eqref{eq0}  with $n=(p-1)^2$ and
$d=(p-2)$
$$
X^2+2p^{3/2}X-\frac{p|A||B||C|}m >0,
$$ 
hence
$$
X> \left(p^3+\frac{p|A||B||C|}m\right)^{1/2}-p^{3/2}.
$$ 
By the following  easy bound
which holds for any real number $a,b>0$
$$
\sqrt{a+b}-\sqrt{a}\ge\frac{b}{2\sqrt{a+b}}\ge\frac{b}{2\sqrt{2}\max(\sqrt{a},\sqrt{b})}
=\frac1{2\sqrt{2}}\min\left(\frac{b}{\sqrt{a}},\sqrt{b}\right),
$$
we get
$$
X^2\gg \min\left(\frac{(|A||B||C|)^2}{pm^2},\frac{p|A||B||C|}{m}\right)
$$
where the implied constant can be taken equal to $1/8$.

By using $|S|\le |f(A,B)||B\cdot C|$ and $|T|\le |A||C|$, we obtain the desired  lower bound 
$$
|f(A,B)||B\cdot C|\ge \frac18\min\left(\frac{|A||B|^2|C|}{pm^2},
\frac{p|B|}{m}\right),
$$
ending the proof of Theorem \ref{t1}.
\end{proof}

By the Freiman Theorem, the condition $|A\cdot A|<K|A|$ implies 
that $A$ is a big subset of a generalized geometric progression. 
It can be shown that for this kind of sets $A$ and $f(x,y)=x(x+y)$ then $|f(A,A)|$ is big.
More generally, we can deduce from Theorem \ref{t1} a conditional expanding property for
the function $f(x,y)=g(x)(h(x)+y)$.

\begin{cor} Let $f(x,y)=g(x)(h(x)+y)$ and $A\subset\bFp^*$ such that $|A\cdot A|\ll |A|^{1+\theta}$ where 
$\theta<1/2$. Then
$$
|f(A,A)|\gg \min\left(\frac{|A|^{3-\theta}}{pm^2},\frac{p|A|^{-\theta}}{m}\right),
$$
where $m=\mu(g\cdot h)$.
\end{cor}
The previous bound is non trivial whenever $|A|^{2-\theta}>pm^2$ and $p>m|A|^{1+\theta}$.
For instance, if $|A|=p^{3/5}$ and $m=O(1)$, then $|f(A,A)|\gg |A|^{4/3-\theta}$.  For $\theta=1/9$,
it shows
$$
\max(|A\cdot A|,|f(A,A)|)\gg |A|^{1+1/9}.
$$
The previous bound holds when $g(x)=x$, \textit{i.e.} $f(x,y)=x(x+y)$. In that case $m=2$ and we have

\begin{cor} 
Let $f(x,y)=x(x+y)$ and $A\subset \bFp^*$. Then 
$$
\max(|f(A,A)|,|A\cdot A|)\gg 
\min\left(\frac{|A|^2}{\sqrt{p}},\sqrt{p|A|}\right).
$$
\end{cor}
We may notice here that in \cite{Sh}, an unconditional sharp result
has been obtained for the set 
$A\circ B=\{a(a+b)\,:\, (a,b)\in A\times B\}$: if $\ln(|A||B|) >
 \frac32\ln p$
then $|A\circ B|\ge p+o(p)$.

\medskip
Another application is for $A$ being a set of $G$ the multiplicative subgroup of the non zero squares of $\bFp^*$. We fix a set of residues  
$R=\{a_1,\dots,a_{(p-1)/2}\}$ such that $a_i\pm a_j\ne0$ for all $i\ne j$. 
we define $g:G\to R$ to be a square root map: for $x\in R$,
$g(x^2)=x.$

We apply our result to $A,B\subset G=\{a_1^2,\dots,a_{(p-1)/2}^2\}$ and
$f(x,y)=g(x)(x+y)$. In our graph, for $x^2,y^2\in A$ and $z^2\in B$, the vertices $(z^2g(x^2z^2)/g(x^2),x^2z^2)$ and $(g(x^2)(x^2+y^2),y^2z^2)$  are connected since
$$
\frac{z^2g(x^2z^2)}{g(x^2)}g(x^2)(x^2+y^2)=g(x^2z^2)((zx)^2+(zy)^2)
=g((zx)^2)((zx)^2+(zy)^2).
$$
We also note that here $\mu(g\cdot\mathrm{id})\le 3$ since $x^3=t$ has at most $3$ different solutions.
We thus have proved the following result.

\begin{cor}
Let $f(x,y)=x(x^2+y^2)$ and $A\subset \bFp^*$. Then 
$$
\max(|f(A,A)|,|A\cdot A|)\gg 
\min\left(\frac{|A|^2}{\sqrt{p}},\sqrt{p|A|}\right).
$$
\end{cor}
We can obtain  a similar result with $k$-th powers instead of squares.
These two above corollaries have to be compared with 
Theorem \ref{t5} which will be proved in section \ref{s7}. The bound
is in the spirit of  Garaev's bound when $f(x,y)=x+y$ in the background 
of the sum-product problem.

\section{\bf Conditional expansion by the function $\boldsymbol{f(x,y)=g(x)(h(x)+y)}$ relatively to the sum function $\boldsymbol{x+y}$}
\label{s3}

Here again, $G$ denotes a subgroup of $\bFp^*$, $A$ is a subset of $G$
and $B,C$ are subsets of $\bFp^*$. Moreover $g$ and $h$ are arbitrary functions from
$G$ into $\bFp^*$. 
We recall that $\mu(g)$ (resp. $\mu(h)$)  denotes the maximal number of solutions $x\in G$ of the equation $t=g(x)$ (resp. $t=h(x)$).

We now rely  the function $f(x,y)$  to the sum-product  graph
$\Gamma$: two vertices $(a,b)$ and$(c,d)$ are connected if $ac=b+d$.
We argue as in the previous section by applying  Theorem\ref{th1}.

We let
\begin{align*}
S&=\{(g(x)(h(x)+y),y+z)\,:\, (x,y,z)\in A\times  B \times C\},\\
T&=\{(g(x)^{-1},h(x)-z)\,:\, x\in A,\ z\in C\}.
\end{align*}
Then $|S|\le |f(A,B)||B+C|$ and $|T|\le |A||C|$. 

\noindent
In order to estimate $e(S,T)$, we need an upper bound for  the number of solutions to the system
$$
\begin{cases}
g(x)(h(x)+y)=u\\
h(x)-z=v\\
g(x)=w\\
y+z=t.
\end{cases}
$$
There are at most $\mu(g)$ admissible values for $x$. 
Now $z$ and $y$ are uniquely determined by $x$.
We thus have 
$$
e(S,T)\ge \frac{|A||B||C|}{\mu(g)}.
$$

\noindent
Applying  Theorem \ref{th1} we get 
$$
|f(A,B)||A||C||B+C|\gg \min\left(\frac{p|A||B||C|}{\mu(g)},\frac{(|A||B||C|)^2}{p\mu(g)^2}
\right),
$$
hence
$$
|f(A,B)||B+C|\gg \min\left(\frac{p|B|}{\mu(g)},\frac{|A||B|^2|C|}{p\mu(g)^2}
\right)
$$
ending the proof of  Theorem \ref{t2}.\qed

\medskip
If we take $h(x)=x$,  by combining Theorems \ref{t1} and \ref{t2}, we get

\begin{cor} 
Let $f(x,y)=g(x)(x+y)$  such that $\mu(g)=O(1)$ and $A\subset \bFp^*$. Then 
$$
|f(A,A)|\times\min(|A\cdot A|,|A+A|)\gg 
\min\left(\frac{|A|^4}{p},p|A|\right).
$$
\end{cor}

\section{\bf A generalisation - Proof of Theorem \ref{t3}}
\label{s4}

We now consider expansion by $f(x,y)=g(x)h(y)(x+y)$ where $g$ and $h$ are 
functions defined on some subgroup $G$ of $\bFp^*$. One defines 
the graph $\Gamma$: the vertices are the couples $(a,b)\in
\bFp^*\times \bFp^*$ and
two vertices $(a,b)$ and $(c,d)$ are connected in $\Gamma$ if 
$$
ac=f(b,d).
$$
The coefficients of the matrix $\mathstrut^{t}\! MM$ where $M$ is the adjacency matrix
of the graph $\Gamma$ have been computed in Section \ref{s2}
(just think $h$ instead of $g'$).\\
In $\Gamma$, the vertices $(f(x,y),yz)$ and $(zg(xz)h(yz)g(x)^{-1}h(y)^{-1},xz)$ are connected.  Let
\begin{align*}
S&=\{(g(x)h(y)(x+y),yz)\,:\, x\in A,\ y\in B,\ z\in C\},\\
T&=\{(zg(xz)h(yz)g(x)^{-1}h(y)^{-1},xz)\,:\, x\in A,\ y\in B,\ z\in C\}.
\end{align*}
Then $|S|\le |f(A,B)||B\cdot C|$ and $|T|\le m|C||A\cdot C|$, where
$$
m=\max_z|\{zg(xz)h(yz)/g(x)h(y)\,:\, x,y\in G\}|.
$$ 
The number of edges between $S$ and $T$ is greater than 
$|A||B||C|/m'$ where $m'$ is the maximal number 
of solutions to the system
$$
\begin{cases}
g(x)h(y)(x+y)=u\\
yz=v\\
zg(zx)h(zy)=wg(x)h(y)\\
zx=t,
\end{cases}
$$
with fixed $(u,v,w,t)$. We get $xg(x)h(vx/t)=tg(t)h(v)/w$. If the  number 
of $x$ satisfying this identity is at most $m''$, then we obtain $m''\ge m'$ and
$e(S,T)\ge |A||B||C|/m''$.

The proof of Theorem \ref{t3} with $k=1$ can be easily read from section \ref{s2}.

If we choose $G$ to be the non zero $k$-th powers, and by letting
$f(x,y)=g(x)h(y)(x^k+y^k)$, we get
$$
|f(A,B)||A\cdot C||B\cdot C|\gg
\min\left(\frac{|A|^2|B|^2|C|}{p},p|A||B|\right)
$$
where the implied constant depends on $m$ and $k$.

Let us consider the particular case $f(x,y)=x^uy^v(x+y)$ where $u$ and $v$ are
fixed positive integers. 
The construction of the convenient graph with many edges reduces to the sets
\begin{align*}
S&=\{(x^uy^v(x+y),yz)\,:\, x\in A,\ y\in B,\ z\in C\},\\
T&=\{(z^{u+v+1},xz)\,:\, x\in A,\ y\in B,\ z\in C\}
\end{align*}
which satisfy 
 $|S|\le |f(A,B)||B\cdot C|$ and $|T|\le |A||C|$. Hence we have
\begin{equation}\label{nnn1}
|f(A,B)||B\cdot C|\gg 
\min\left(\frac{|A||B|^2|C|}{p},p|B|\right).
\end{equation}

\begin{rem}
When $|A|,|B|\asymp p^{\alpha}$ with $2/3\le \alpha<1$, the above result is sharp. Indeed,
letting $A=B=C$ being a geometric progression of length $\asymp p^{\alpha}$, we plainly have
$|A\cdot A|\ll |A|$ and $|f(A,A)|\le p$ hence $|f(A,A)||A\cdot A|\ll p|A|$.
\end{rem}

\begin{rem}
Letting $C=A$, we get
$$
\max(|f(A,B)|,|A\cdot B|)\gg \min\left(\frac{|A||B|}{\sqrt{p}},\sqrt{p|B|}\right).
$$
This yields the following
$$
|A|,|B|\asymp p^{\alpha}\Rightarrow 
\max(|f(A,B)|,|A\cdot B|)\gg|A|^{1+\Delta(\alpha)}
$$
with $\Delta(\alpha)=\min(1-1/2\alpha,(1/\alpha-1)/2)$.

It is not known whether or not the above bound is sharp, namely if there exists
sets $A$ and $B$ such that  $|A\cdot B|,|f(A,B)|\ll \min\left(\frac{|A||B|}{\sqrt{p}},\sqrt{p|B|}\right)$ for $|A|\asymp|B|\asymp p^{\alpha}$
and $0<\alpha<1$.
\end{rem}


The bounds we obtained are non trivial for $\alpha>1/2$.  
For smaller $\alpha$ we can prove  the following explicit bound:

\begin{thm} \label{t7}
Let $A\subset \bFp^*$ such that $|A|\le p^{1/2-1/500}$. Let $f(x,y)=xy(x+y)$. Then
$$
\max(|f(A,A)|,|A\cdot A|)\gg|A|^{1+1/800}.
$$
\end{thm}

\begin{proof}
In the proof, the $C_i$'s will denote positive absolute constants.
We write $f(x,y)=x^2y+xy^2$ and we define $K$ by
$$
\max(|f(A,A)|,|A\cdot A|)=K|A|.
$$
By Pl\"unnecke-Ruzsa's  inequality for triple product in a commutative group,  we get $|A\cdot A\cdot A|\le K^{3}|A|$.  We denote by $A^{(k)}$ the set formed by the $k$-th powers  of the elements of $A$.
Let $\Gamma$ be the graph
$$
\Gamma=\{(u,v)\in (A^{(2)}\cdot A)\times  (A^{(2)}\cdot A)\,:\ (u^2/v,v^2/u)\in A^{(3)}\times A^{(3)}  \}.
$$ 
Then $\Gamma\subset (A^{(2)}\cdot A)\times  (A^{(2)}\cdot A)\subset (A\cdot A\cdot A)\times (A\cdot A\cdot A)$ and 
$$
|\Gamma|=|A^{(3)}|^2\gg|A|^2\ge |A\cdot A\cdot A|^2/K^6\ge |A^{(2)}\cdot A|^2/K^6.
$$ 
For $X,Y\subset A^{(2)}\cdot A$, we denote
the restricted sumset 
$$X\overset{\Gamma}{+} Y=\{x+y\,:\, x\in X,\, y\in Y,\, (x,y)\in\Gamma\}.
$$

\noindent
Since $|A^{(2)}\cdot A\overset{\Gamma}{+}A^{(2)}\cdot A| =|f(A,A)|\le K|A|$, by the Balog-Szemer\'edi-Gowers Theorem
 (see \cite[Theorem 2.29]{TV} where we change the exponent $K^4$ to $K^5$), there exists two subsets $A_1,B_1\subset A^{(2)}\cdot A $ 
 such that $|A_1|,|B_1|\ge C_1|A^{(2)}\cdot A|/K^6
\ge C_1|A|/K^6 $ and
$$
|A_1 + B_1|\le C_2K^{33}|A_1|^{1/2}|B_1|^{1/2}\le 
C_2 K^{33}|A_1|,
$$
by assuming (wlog) that $|A_1|\ge|B_1|$. By  Ruzsa's triangle inequality, we get $|A_1+A_1|\le C_3K^{66}|A_1|$. \\
We deduce that if  $K\ll |A|^{1/800}$, then $K\ll |A_1|^{1/794}$ and $|A_1+A_1|\ll |A_1|^{13/12}$. 
By  Pl\"unnecke-Ruzsa inequalities, we get $|A_1|\le |A\cdot A\cdot A|\le K^3|A|\le  |A|^{1+3/800}\le p^{1/2}$ by our assumption on $|A|$,  and by the sum-product estimate (see \cite{BT}) 
\begin{equation}\label{nnn4}
|A_1+A_1|+|A_1\cdot A_1|\gg |A_1|^{13/12}.
\end{equation}
Using $A_1\cdot A_1\subset A\cdot A\cdot A\cdot A\cdot A\cdot A$ and  Pl\"unnecke-Ruzsa inequalities again,
we must have 
$$
K^6|A|\ge |A\cdot A\cdot A\cdot A\cdot A\cdot A|\ge
|A_1\cdot A_1|\ge C_4|A_1|^{13/12}\ge \frac{C_5|A|^{13/12}}{K^{13/2}},
$$
yielding $K\gg|A|^{1/150}$, a contradiction.
Hence we have $K\gg |A|^{1/800}$,
 ending the proof.
\end{proof}
We did not optimize the expanding exponent $1/800$ in the above theorem, but according to the best known sum-product estimates, it is close to the best possible exponent that we can obtain by this method for small subset $A$ of $\mathbb{F}_p$.

\medskip
One can ask the question of obtaining a similar result with a pair of sets $A,B$. 
There are functions $f(x,y)$ which are not expanders but such that
$|f(A,A)|\gg |A|^{1+\theta}$ if $p^{\delta}<|A|<p^{1-\delta}$,
\textit{e.g.} $f(x,y) =x(1+y)$
is such a function.

The same arguments used in the proof of the above theorem, with an appropriate and key 
adaptation when applying Plünnecke-Ruzsa inequalities, allow us to show:

\begin{thm}\label{t8}
Let $A,B\subset \bFp^*$ such that $|A|,|B|\le p^{1/2-1/400}$ and $|A|\asymp|B|$. Let $f(x,y)=xy(x+y)$. Then
$$
\max(|f(A,B)|,|A\cdot B|)\gg|A|^{1+\theta}
$$
for some $\theta>0$.
\end{thm}

\begin{proof}
We will not try to find the best exponent $\theta$.
We let
$$
\max(|f(A,B)|,|A\cdot B|)=K|A|.
$$
By the quantitative Pl\"unnecke-Ruzsa inequalities (cf. Theorem 1.7.3 of \cite{Ru} with $t=m/2$), there exists $A'\subset A$ such that
$|A'|\ge|A|/2$ and $|A'\cdot B\cdot B|\le 6K^2|A'|\ll K^2|B|$. Applying again these inequalities, there exists
$B'\subset B$  and $B''\subset B'$ such that 
$$|B'|\ge|B|/2\quad \text{and}\quad 
|B'\cdot B\cdot A'\cdot B\cdot A'|\ll K^4|B'|
$$ 
and
$$
|B''\cdot B\cdot A' \cdot B\cdot A' \cdot B\cdot A' \cdot B\cdot A'|\ll K^8|B''|.
$$
It follows that for these sets $A'$ and $B'$ we have
$$
|B'\cdot B'\cdot A'|\ll K^2|A'|\ \text{and} \ 
|B'\cdot A'\cdot A'|\ll K^4|B'| \ \text{and} \ 
|B\cdot A' \cdot B\cdot A' \cdot B\cdot A' \cdot B\cdot A'|\ll K^8|A'|.
$$
Furthermore $\max(|f(A',B')|,|A'\cdot B'|)\le 2K|A'|$.
Let $\Gamma$ be the graph
$$
\Gamma=\{(u,v)\in (A'^{(2)}\cdot B')\times  (B'^{(2)}\cdot A')\,:\ (u^2/v,v^2/u)\in A'^{(3)}\times B'^{(3)}  \}.
$$
Then $\Gamma\subset (A'^{(2)}\cdot B')\times  (B'^{(2)}\cdot A')\subset (A'\cdot A'\cdot B')\times (A'\cdot B'\cdot B')$ and 
$$
|\Gamma|\gg|A'||B'|\ge |A'\cdot A'\cdot B'||A'\cdot B'\cdot B'|/K^{6}\ge |A'^{(2)}\cdot B'||A'\cdot B'^{(2)}|/K^{6}.
$$ 
Since $|A'^{(2)}\cdot B'\overset{\Gamma}{+}B'^{(2)}\cdot A'| =|f(A',B')|\le2 K|A'|$, by the Balog-Szemer\'edi-Gowers Theorem
 (see \cite[Theorem 2.29]{TV}), there exists two subsets $A_1\subset A'^{(2)}\cdot B'$ and  
$B_1\subset B'^{(2)}\cdot A'$
 such that
$$
 |A_1|,|B_1|\ge C_1|A'|/K^{6}\quad \text{and}\quad |A_1 + B_1|\le C_2K^{33}|A_1|^{1/2}|B_1|^{1/2}\le 
C_2 K^{33}\max(|A_1|,|B_1|).
$$ 
Assume that $|A_1|\ge|B_1|$ (in the opposite case, we argue similarly by considering $B_1$ instead of $A_1$).  Then
by  Ruzsa's triangle inequality, we get
$$
|A_1+A_1|\le C_3K^{66}|A_1|.
$$
We deduce that if  $K\ll |A'|^{1/800}$, then $K\ll |A_1|^{1/794}$ and $|A_1+A_1|\ll |A_1|^{13/12}$. Since $A_1\cdot A_1\subset A'\cdot A'\cdot A'\cdot A'\cdot B'\cdot B'$ and $|A_1|\le |A'\cdot A'\cdot B'|\ll K^4|A'|\ll |A'|^{1+1/200}\le\sqrt{p}$,
the sum-product estimate \eqref{nnn4} holds and
$$
K^8|A'|\gg |A'\cdot A'\cdot A'\cdot A'\cdot B'\cdot B'|\ge
|A_1\cdot A_1|\ge C_4|A_1|^{13/12}\ge \frac{C_5|A'|^{13/12}}{K^{13/2}},
$$
yielding $K\gg|A'|^{1/174}$, a contradiction.
Since $|A'|\gg|A|$, 
it shows that we can choose $\theta=1/800$ in the theorem.
\end{proof}

\begin{rem}
A similar result  with  $f(x,y)=x^uy^v+x^{u'}y^{v'}$ where the positive integers $u,v,u',v'$ are
fixed,  can be proved in the same manner (with a weaker expanding exponent). 
\end{rem}

\section{\bf Conditional growth for functions and their shifted functions}
\label{s7}
Here we prove Theorems \ref{t5} and \ref{t6}. 

 \begin{proof}[Proof of Theorem \ref{t5}]
It is enough to consider the case where 
$f(x,y)=g(x)g'(y)(h(x)+y)$ and $P(f)(x,y)=xf(x,y)$ since the general case 
will follow by applying the simplest case with $w(A)$ instead of $A$
to the function 
$$
f(x,y)=g(w^{-1}(x))g'(y)(h(w^{-1}(x))+y),
$$
where $w^{-1}(x)$ is some preimage of $x$ by $w$. Our hypothesis
$\mu(w)=O(1)$ insures that $|w(A)|\gg|A|$.

Denote $\max(|f(A,A)|,|P(f)(A,A)|)=K|A|$. Let
$$
\Gamma=\{(u,v)\in A\times f(A,A)\,:\, v\in f(u,A)\},
$$
and denote
$X\overset{\Gamma}{\cdot}Y=\{xy\,:\, (x,y)\in X\times Y\}$
for $X,Y\subset A$.

The constant $C_i$ appearing in the sequence are positive
real numbers depending only on $f$.

We plainly have $|\Gamma|\ge C_0 |A|^2$, where $C_0$ depends only on $f$.
 Since $|A\times f(A,A)|\le K|A|^2$ and 
$|A\overset{\Gamma}{\cdot}f(A,A)|\le K|A|$, we get 
from the Balog-Szemer\'edi-Gowers Theorem (see \cite[Theorem 2.29]{TV}) that there is $A_1\subset A$ and $B_1\subset f(A,A)$ such that $|A_1|\ge C_1|A|/K$,  $|B_1|\ge C_1|A|/K$
and
$$
|A_1\cdot B_1|\le C_2K^8|A_1|^{1/2}|B_1|^{1/2}\le 
C_3 K^{9}|A_1|.
$$
By the Ruzsa  triangle inequality, 
$|A_1\cdot A_1|\le C_4 K^{18}|A_1|$.
By the result from the preceding sections (cf. Theorem \ref{t3}), we now can prove that $f(A_1,A_1)$ is big: this gives 
\begin{equation}\label{d1}
|f(A_1,A_1)|\ge \frac{C_5}{K^{36}}\min\left(\frac{|A_1|^3}{p},p\right) 
\end{equation}
hence 
\begin{equation}\label{d2}
K|A|\ge |f(A,A)|\ge \frac{C_6}{K^{39}} \min\left(\frac{|A|^3}{p},p\right).
\end{equation}
If $|A|\asymp p^{\alpha}$, then
we get $K\gg |A|^{\Delta(\alpha)}$ where
\begin{equation}\label{d3}
\Delta(\alpha)=\frac{\min(2-1/\alpha,1/\alpha-1)}{40}.
\end{equation}
For instance if $\alpha=2/3$ it yields the bound 
$$
\max(|f(A,A)|,|P(f)(A,A)|)\gg |A|^{81/80}.
$$
\nopagebreak
\vspace{-1.2cm}\[\qedhere\]
\end{proof}

\begin{proof}[Proof of Theorem \ref{t6}]
We argue similarly for the additive shifted function starting from $f(x,y)=g(x)(h(x)+y)$ in order to obtain a joint expanding lower bound for $f(x,y)$ 
and $S_w(f)(x,y)=w(x)+f(x,y)$. 

As in the product case, we may assume that $w(x)=x$ 
and we let 
$$\max(|f(A,A)|,|S(f)(A,A)|)=K|A|.
$$
In the above proof, we may replace product by sum in such a way
that we get  
by Balog-Szemer\'edi-Gowers Theorem, 
the existenceness of $A_1\subset A$ such that
$|A_1+A_1|\le C_4K^{18}|A_1|$ with $|A_1|\ge C_1|A|/K$.
By Theorem \ref{t2}, \eqref{d1}, \eqref{d2} and \eqref{d3} remain all valid
and Theorem \ref{t6} follows.
\end{proof}


\section{\bf Expansion in the real numbers}
\label{s6}

We do not investigate above the real expansion of sets of real numbers by functions $f(x,y)$. 
But it seems interesting to show that 
Solymosi's geometric approach of the sum-product problem (cf. \cite{So2}) applies for studying the expanding size
of $f(A,A)$  for functions like
$f(x,y)=xy(x^{k}+y^k)$ where 
$k\ne 0$ is a real number. 

Let $A$ be a finite set of real positive numbers. We first assume that 
$k>0$. We define in $\mathbb{R}^2$ the following product rule:
$$
(x,y)*(x',y')=(f(x,x'),f(y,y')).
$$
Let $E(A)$ the multiplicative energy of $A$. Then by usual means, we obtain
\begin{equation}\label{nnn3}
E(A)\ge\frac{|A|^4}{|A\cdot A|}.
\end{equation}
For $\alpha\in A/A=\{y/x\,:\, x,y\in A\}$,
we let $A_{\alpha}=\{(x,y)\in A\times A\,\mid\, y=\alpha x\}$. Hence
$E(A)=\sum_{\alpha}|A_{\alpha}|^2$. We denote $E_i=\{\alpha\in A/A\,:\, 2^{i}\le |A_{\alpha}|<2^{i+1}\}$.
We may 
write
$$
E(A)=\sum_{i\ge0}\sum_{\alpha\in E_i}|A_{\alpha}|^2.
$$
Hence there exists $i$ such that 
$$
\sum_{\alpha\in E_i}|A_{\alpha}|^2\ge \frac{E(A)}{\log|A|}.
$$
We put $d:=|E_i|$. Then 
\begin{equation}\label{nnn2}
4^{i+1}d\ge \frac{E(A)}{\log|A|}.
\end{equation}
We arrange the elements of $E_i$ in the increasing order,
$\alpha_1<\cdots<\alpha_d$ and we let $A_{\alpha_{d+1}}=\{a_0\}\times A$ where $a_0=\min A$.  We get
$$
|f(A,A)|^2=|(A\times A)*(A\times A)|\ge\Big|\bigcup_{j=1}^dA_{\alpha_j}*A_{\alpha_{j+1}}\Big|.
$$
Now notice that for $\alpha<\alpha'$, and $(x,y)\in A_{\alpha}$,  $(x',y')\in A_{\alpha'}$, we have
$$
\frac{f(y,y')}{f(x,x')}=\frac{f(\alpha x,\alpha' x')}{f(x,x')}=
\frac{\alpha^{1+k}\alpha'x^{1+k}x'+\alpha\alpha'^{1+k}xx'^{1+k}}{x^{1+k}x'+xx'^{1+k}}.
$$
It shows that $(x,y)*(x',y')$ is on the line passing by the origin with slope between $\alpha^{1+k}\alpha'$ and
$\alpha\alpha'^{1+k}$. Since $u$ is increasing, we have
for $\alpha<\alpha'<\alpha''$,  $\alpha^{1+k}\alpha'<\alpha\alpha'^{1+k}
<\alpha'^{1+k}\alpha''$,  hence 
$|f(A,A)|^2\ge \sum_{j=1}^d|A_{\alpha_j}*A_{\alpha_{j+1}}|$.
 Moreover another couple $((z,t);(z',t'))\in A_{\alpha}\times A_{\alpha'}$ gives a different point
$(z,t)*(z',t')$. Indeed if it coincides with $(x,y)*(x',y')$, then necessarily $z^{1+k}z'=x^{1+k}x'$ and $zz'^{1+k}=xx'^{1+k}$. Hence 
$z=x$ and $z'=x'$.

We thus deduce by \eqref{nnn2} and \eqref{nnn3}
$$
|f(A,A)|^2\ge 
\sum_{j=1}^d|A_{\alpha_j}||A_{\alpha_{j+1}}|\ge 4^id> \frac{E(A)}{4\log|A|}
\ge  \frac{|A|^4}{4|A\cdot A|\log|A|}.
$$
For getting the result when $k<0$, it suffices to 
consider $A_{1/\alpha}$ instead
of $A_{\alpha}$ at the beginning of the above proof. 
We thus have shown the following result.

\begin{prop}\label{pp71} Let $f(x,y)=xy(x^k+y^k)$ where $k$ is non zero real number.
Let $A$ be a finite set of positive real numbers. Then
$$
\max(|f(A,A)|,|A\cdot A|)\gg \big(|A|^{4}/\log|A|\big)^{1/3}.
$$
\end{prop}

\begin{rem}
Notice that applying Proposition \ref{pp71} with $k=-1$
gives back the Solymosi's sum-product estimate
$SP(A)\gg\big(|A|^{4}/\log|A|\big)^{1/3}$.
\end{rem}

In the particular case $f(x,y)=xy(x+y)$, the above bound is superseded by the following non conditional result.

\begin{prop}\label{pp73}
 Let $f(x,y)=xy(x+y)$ and $A,B$ be two non empty finite sets of non zero real numbers. Then
\begin{equation}\label{eqhh}
|f(A,B)|\gg |A|^{2/3}|B|^{2/3}.
\end{equation}
\end{prop}

\begin{proof} We plainly may assume that 
\begin{equation}\label{eqkk}
|A|^{1/2}\le |B|\le |A|^2
\end{equation}
since otherwise the bound \eqref{eqhh} clearly holds by the
easy fact that $|f(A,B)|\gg\max(|A|,|B|)$.
 
For each $(a,b)\in A^2$, we let
$$
\gamma_{a,b}=\{(y,y')\in \mathbb{R}^2\,\mid\, ay^2+a^2y=by'^2+b^2y'\},
$$ 
and 
$$
\mathcal{P}=\{(y,y')\in B^2\,:\,y^3\ne y'^3\}
,\quad \mathcal{C}=\{\gamma_{a,b}\,:\, (a,b)\in A^2,\ a^3\ne b^3\}.
$$
One has
\begin{equation}\label{eqdd}
\sum_{(a,b)\in A^2}|\gamma_{a,b}\cap B^2|
=\sum_{\substack{(a,b)\in A^2\\a^3 =b^3}}|\gamma_{a,b}\cap B^2|
+\sum_{\substack{(a,b)\in A^2\\a^3\ne b^3}}|\gamma_{a,b}\cap (B^2\setminus \mathcal{P})|
+\sum_{\substack{(a,b)\in A^2\\a^3\ne b^3}}|\gamma_{a,b}\cap \mathcal{P}|.
\end{equation}
For fixed $a\in A$, $\xi$ a cubic root of unity and $y\in B$, the equation $ay^2+a^2y=a\xi y'^2+a^2\xi^2y'$ has
at most $2$ solutions $y'\in B$, hence
the first sum in the right-hand side of \eqref{eqdd} is $\le 6 |A||B|$.
From the fact that 
\begin{equation}\label{eqgg}
(y,y')\in\gamma_{a,b}\iff (a,b)\in\gamma_{y,y'},
\end{equation}
it is not difficult to deduce  that the second sum in \eqref{eqdd}
is less
$$
\sum_{\substack{(y,y')\in B^2\\y^3 =y'^3}}|\gamma_{y,y'}\cap A^2|
$$
which is bounded in $6|A||B|$ in the same way as above.

The third sum in the right-hand side of \eqref{eqdd} counts the
number of pairs $(\pi,\gamma)\in \mathcal{P}\times
\mathcal{C}$ such that $\pi\in \gamma$, that is the number of incidences of points in $\mathcal{P}$ on the curves of $\mathcal{C}$. 
We now check that two different curves of $\mathcal{C}$ intersect in at most  three points 
and that two different points of $\mathcal{P}$ are simultaneously incident to at most three curves of $\mathcal{C}$. These two 
statement are essentially equivalent by \eqref{eqgg}. It is thus sufficient  to show that
if $(a,b)\ne(c,d)$ and $a^3\ne b^3$ then 
\begin{equation}\label{eqee}
|\gamma_{a,b}\cap\gamma_{c,d}|\le3.
\end{equation}
A point $(y,y')\in(\mathbb{R}^*)^2$ belongs to $\gamma_{a,b}\cap\gamma_{c,d}$ if and only if
$$
ay^2+a^2y=by'^2+b^2y'\quad\text{and}\quad
cy^2+c^2y=dy'^2+d^2y'.
$$
By letting $z=y'/y$ one gets
\begin{equation}\label{eqff}
(a-bz^2)y=(b^2z-a^2)\quad\text{and}\quad
(c-dz^2)y=(d^2z-c^2),
\end{equation}
hence $(a-bz^2)(d^2z-c^2)=(c-dz^2)(b^2z-a^2)$ or equivalently
$bd(b-d)z^3+(a^2d-bc^2)z^2+(b^2c-ad^2)z+ac(a-c)=0$. This equation has at most 3 solutions $z\in \mathbb{R}$ if $b\ne d$ and at most 2 solutions if $b=d$ and $a\ne c$. Each solution $z$ yields 
at most one solution $y$ to the system \eqref{eqff}, except possibly 
when
$a^3=b^3$, $c^3=d^3$ and $ab=cd$ in which case
$z=a^2/b^2$ is a cubic root of unity.
 Whence \eqref{eqee}.

By the generalized Szem\'eredi-Trotter type Theorem 8.10
of \cite{TV}  it follows that
$$
\sum_{\substack{(a,b)\in A^2\\a^3\ne b^3}}|\gamma_{a,b}\cap \mathcal{P}|
\ll |\mathcal{P}|+|\mathcal{C}|+(|\mathcal{P}||\mathcal{C}|)^{2/3}
\le
|A|^2+|B|^2+(|A|^2|B|^2)^{2/3}.
$$
For $z\in f(A,B)$ let $r(z)$ be the number of pairs $(a,y)\in A\times B$
such that $z=ay^2+a^2y$. Hence
by Cauchy-Schwarz inequality we get
\begin{align*}
|f(A,B)|&\ge \frac{\left(\sum_{z}r(z)\right)^2}{\sum_{z}r(z)^2}
=\frac{|A|^2|B|^2}{\sum_{(a,b)\in A^2}|\gamma_{a,b}\cap B^2|}\\
&\gg \frac{|A|^2|B|^2}{12|A||B|+|A|^2+|B|^2+
(|A||B|)^{4/3}}\gg (|A||B|)^{2/3}\quad\text{by \eqref{eqkk},}
\end{align*}
as asserted.
\end{proof}

\end{document}